\documentclass[11pt]{article}

\usepackage{amsmath,amsthm}
\usepackage{amssymb}
\usepackage{hyperref}

\allowdisplaybreaks

\begin{document}

\title{Generalized hyperharmonic number sums with reciprocal binomial coefficients}
\author{
Rusen Li
%\thanks{
%}
\\
%\small Department of Mathematical Sciences, School of Science\\
%\small Zhejiang Sci-Tech University\\
%\small Hangzhou 310018 China\\
\small School of Mathematics\\
\small Shandong University\\
\small Jinan 250100 China\\
%\small \texttt{komatsu@zstu.edu.cn}\\
\small \texttt{limanjiashe@163.com}
}

\date{
\small 2020 MR Subject Classifications: 05A10, 11B65, 11B68, 11B83, 11M06
%\small Submitted: November 15, 2019;  Accepted: December 25, 2019.\\
%\small MR Subject Classifications: Primary 11B65; Secondary 05A19
}

\maketitle

\def\stf#1#2{\left[#1\atop#2\right]}
\def\sts#1#2{\left\{#1\atop#2\right\}}
\def\e{\mathfrak e}
\def\f{\mathfrak f}

\newtheorem{theorem}{Theorem}
\newtheorem{Prop}{Proposition}
\newtheorem{Cor}{Corollary}
\newtheorem{Lem}{Lemma}
\newtheorem{Example}{Example}
\newtheorem{Remark}{Remark}
\newtheorem{Definition}{Definition}
\newtheorem{Conjecture}{Conjecture}
\newtheorem{Problem}{Problem}

\begin{abstract}
In this paper, we mainly show that generalized hyperharmonic number sums with reciprocal binomial coefficients can be expressed in terms of classical (alternating) Euler sums, zeta values and generalized (alternating) harmonic numbers.
\\
{\bf Keywords:} generalized hyperharmonic numbers, classical Euler sums, binomial coefficients, combinatorial approach, partial fraction approach
\end{abstract}

\section{Introduction and preliminaries}

Let $\mathbb Z$, $\mathbb N$, $\mathbb N_{0}$ and $\mathbb C$ denote the set of integers, positive integers, nonnegative integers and complex numbers, respectively. In the present paper, we mainly study the so-called generalized hyperharmonic numbers {\cite{Dil,Rusen,omur}} which are defined as
$$
H_n^{(p,r)}:=\sum_{j=1}^n H_{j}^{(p,r-1)} \quad (n,p,r \in \mathbb N),
$$
where $H_n^{(p,1)}=H_n^{(p)}=\sum_{j=1}^n 1/{j^p}$ are the well studied classical harmonic numbers. Note that, $H_n^{(1,r)}=h_n^{(r)}$ are the classical hyperharmonic numbers introduced by Conway and Guy {\cite{Conway}}. To see combinatorial interpretations of these hyperharmonic numbers and their connections with Stirling numbers, please find Benjamin et al's interesting paper {\cite{Benjamin}}. For convenience, we recall the generalized alternating harmonic numbers which are defined as
$$
\overline{H}_n^{(m)}:=\sum_{j=1}^n \frac{(-1)^{j-1}}{j^m} \quad (n, m \in \mathbb N).
$$

The harmonic numbers and their generalizations has caused many mathematicians' interest (see \cite{Doelder,Dil,Dil1,Flajolet,Kamano,Rusen,LiRusen,Matsuoka,mezo,Sofo,Sofo1,Sofo2} and references therein), since they play an essential role in number theory, combinatorics, analysis of algorithms and many other areas (see e.g. \cite{Knuth}). One of the most famous result that obtained by Euler \cite{Flajolet} is the following identity
$$
2\sum_{n=1}^\infty \frac{H_n}{n^{m}}=(m+2)\zeta(m+1)-\sum_{n=1}^{m-2}\zeta(m-n)\zeta(n+1), \quad m=2,3,\cdots.
$$
It is interesting that the Riemann zeta functions $\zeta(s):=\sum_{n=1}^\infty n^{-s}$ appear in such expressions.
According to the recording of Ramanujan's Notebooks \cite[p.253]{Berndt}, Euler considered this type of infinite series containing harmonic numbers $H_n$ in response to a letter from Goldbach in $1742$.

For convenience, we recall the definition of the well-known Hurwitz zeta function:
$$
\zeta(s,a):=\sum_{n=0}^\infty \frac{1}{(n+a)^{s}} \quad (s\in \mathbb C, \mathfrak Re(s)>1, a>0).
$$
Note that $\mathfrak Re(s)$ denotes the real part of the complex number $s$. When $a=1$, $\zeta(s,1)$ is the famous Riemann zeta function. The alternating zeta function $\overline{\zeta}(s)$ is defined by
$$
\overline{\zeta}(s):=\sum_{n=1}^\infty \frac{(-1)^{n-1}}{n^{s}}=(1-2^{1-s})\zeta(s)\quad (s\in \mathbb C, \mathfrak Re(s)\geq 1),
$$
with $\overline{\zeta}(1)=\log 2$.

From Euler's time on, infinite series containing harmonic numbers or their generalizations have been called Euler sums. It is a difficult task to give explicit evaluation for general Euler sums. Facilitated by numerical computations using an algorithm, Bailey, Borwein and Girgensohn \cite{Bailey} determined, with high confidence, whether or not a particular numerical value involving the generalized harmonic numbers $H_n^{(m)}$ could be expressed as a rational linear combination of several given constants.

Flajolet and Salvy \cite{Flajolet} developed the contour integral representation approach (the most powerful method in the corresponding area as far as the author knows, although restricted to parity principle) to the evaluation of Euler sums involving the classical (alternating) harmonic numbers. Note that, the contour integral representation approach can not only evaluate Euler sums, but also evaluate some infinite series involving hyperbolic functions.

Euler sums of hyperharmonic numbers had also attracted many mathematicians' attention. For instance, Mez\H o and Dil \cite{mezo} considered the Euler sums of type
$$
\sum_{n=1}^\infty \frac{h_n^{(r)}}{n^{m}} \quad (m\ge r+1, m \in \mathbb N),
$$
and showed that it could be reduced to infinite series involving the Hurwitz zeta function values. Later Dil and Boyadzhiev \cite{Dil1} extended this result to infinite series involving multiple sums of the Hurwitz zeta function values.

As a natural generalization, Dil, Mez\H o and Cenkci \cite{Dil} considered Euler sums of generalized hyperharmonic numbers of the form
$$
\zeta_{H^{(p,r)}}(m):= \sum_{n=1}^\infty \frac{H_n^{(p,r)}}{n^{m}}.
$$
They proved that for positive integers $p, r$ and $m$ with $m>r$, $\zeta_{H^{(p,r)}}(m)$ could be reduced to infinite series of multiple sums of the Hurwitz zeta function values. For $r=1, 2, 3$, $\zeta_{H^{(p,r)}}(m)$ were also written explicitly in terms of (multiple) zeta values. Although these results were interesting, Dil et al didn't give general formula for explicit evaluations of Euler sums of generalized hyperharmonic numbers. Fortunately, the author \cite{Rusen} found a combinatorial approach and proved that $\zeta_{H^{(p,r)}}(m)$ could be expressed as linear combinations of classical Euler sums. From Flajolet and Salvy's paper \cite{Flajolet}, we knew that the linear Euler sums
$$
\sum_{n=1}^\infty \frac{H_n}{n^{m}}\quad (m\ge 2, m \in \mathbb N) \quad\hbox{and}\quad
\sum_{n=1}^\infty \frac{H_n^{(p)}}{n^{q}}\quad (p, q \in \mathbb N \quad \hbox{with $p+q$ odd})
$$
could be reduced to zeta values. Thus for small values of $p, r$ and $m$, we can determine the exact values of $\zeta_{H^{(p,r)}}(m)$.

Motivated by Flajolet-Salvy's paper \cite{Flajolet} and Dil-Mez\H o-Cenkci's paper \cite{Dil}, the author \cite{LiRusen} also introduced the notion of the generalized alternating hyperharmonic numbers
$$
H_n^{(p,r,1)}:=\sum_{k=1}^n (-1)^{k-1} H_{k}^{(p,r-1,1)}\quad (H_n^{(p,1,1)}=H_n^{(p)}),
$$
and proved that Euler sums of the generalized alternating hyperharmonic numbers $H_n^{(p,r,1)}$ could be expressed in terms of linear combinations of classical (alternating) Euler sums.

If we regard $\sum_{n=1}^\infty h_n^{(r)}/{n^{s}}$ as a complex function in variable $s$, there are some more progresses toward this direction. For instance, Matsuoka \cite{Matsuoka} proved that $\sum_{n=1}^\infty h_n^{(1)}/{n^{s}}$ admits a meromorphic continuation to the entire complex plane. Kamano \cite{Kamano} expressed the complex variable function $\sum_{n=1}^\infty h_n^{(r)}/{n^{s}}$ in terms of the Riemann zeta functions, and showed that it could be meromorphically continued to the entire complex plane. In addition, the residue at each pole was also given.

There are some more interesting combinatorial properties about the generalized hyperharmonic numbers. For instance, \" Om\" ur and Koparal \cite{omur} defined two $n\times n$ matrices $A_n$ and $B_n$ with $a_{i,j}=H_i^{(j,r)}$ and $b_{i,j}=H_i^{(p,j)}$, respectively, and gave some interesting factorizations and determinant properties of the matrices $A_n$ and $B_n$.

On the contrary, Euler sums of generalized harmonic numbers with reciprocal binomial coefficients had been studied by Sofo. In $2011$, Sofo \cite{Sofo} proved that generalized harmonic number sums with reciprocal binomial coefficients of types
$\sum_{n=1}^\infty\frac{H_n^{(s)}}{\binom{n+k}{k}}$ and $\sum_{n=1}^\infty\frac{H_n^{(s)}}{n\binom{n+k}{k}}$
could be written in terms of zeta values and harmonic numbers. In $2015$, Sofo \cite{Sofo1} developed closed form representations of alternating quadratic harmonic numbers and reciprocal binomial coefficients, including integral representations, of the form
$$
\sum_{n=1}^\infty\frac{(-1)^{n+1}(H_n)^{2}}{n^{p}\binom{n+k}{k}}
$$
for $p=0$ and $1$.
In $2016$, Sofo \cite{Sofo2} developed identities, closed form representations of alternating harmonic numbers of order two and reciprocal binomial coefficients of the form:
$$
\sum_{n=1}^\infty\frac{(-1)^{n+1}H_n^{(2)}}{n^{p}\binom{n+k}{k}}
$$
for $p=0$ and $1$.

In the present paper, we mainly show that generalized hyperharmonic number sums with reciprocal binomial coefficients of types
$$
\sum_{n=1}^\infty\frac{H_n^{(p,s)}}{n^{m}\binom{n+k}{k}}\,,\quad
\sum_{n=1}^\infty\frac{(-1)^{n+1} H_n^{(p,s)}}{n^{m}\binom{n+k}{k}}\,,\\
$$
and
$$
\sum_{n=1}^\infty\frac{H_n^{(p_{1},s_{1})}H_n^{(p_{2},s_{2})}}{n^{m}\binom{n+k}{k}}\,,\quad
\sum_{n=1}^\infty\frac{(-1)^{n+1} H_n^{(p_{1},s_{1})}H_n^{(p_{2},s_{2})}}{n^{m}\binom{n+k}{k}}
$$
can be expressed in terms of linear combinations of classical (alternating) Euler sums, zeta values and generalized (alternating) harmonic numbers. Some illustrative examples are also given. Further more, We give explicit evaluations for some interesting integrals and develop some combinatorial expressions for harmonic numbers in terms of binomial coefficients.

\section{Generalized hyperharmonic number sums}

In this section, we develop closed form representations for generalized hyperharmonic number sums with reciprocal binomial coefficients of types
$$
\sum_{n=1}^\infty\frac{H_n^{(p,s)}}{n^{m}\binom{n+k}{k}}\,\quad\hbox{and}\quad
\sum_{n=1}^\infty\frac{(-1)^{n+1} H_n^{(p,s)}}{n^{m}\binom{n+k}{k}}\,.
$$

Before going further, we introduce some notations and lemmata.

Following Flajolet-Salvy's paper \cite{Flajolet}, we write four types of classical linear (alternating) Euler sums as
\begin{align*}
S_{p,q}^{+,+}:=\sum_{n=1}^\infty \frac{H_n^{(p)}}{{n}^{q}},
\quad S_{p,q}^{+,-}:=\sum_{n=1}^\infty (-1)^{n-1}\frac{H_n^{(p)}}{{n}^{q}},\\
S_{p,q}^{-,+}:=\sum_{n=1}^\infty \frac{\bar{H}_n^{(p)}}{{n}^{q}},
\quad S_{p,q}^{-,-}:=\sum_{n=1}^\infty (-1)^{n-1}\frac{\bar{H}_n^{(p)}}{{n}^{q}}.
\end{align*}

We now recall Faulhaber's formula on sums of powers. It is well known that the sum of powers of consecutive intergers $1^k+2^k+\cdots+n^k$ can be explicitly expressed in terms of Bernoulli numbers or Bernoulli polynomials. Faulhaber's formula can be written as
\begin{align}
\sum_{\ell=1}^{n}\ell^{k}
&=\frac{1}{k+1}\sum_{j=0}^k \binom{k+1}{j}B_j^{+} n^{k+1-j}\label{ber}\\
&=\frac{1}{k+1}(B_{k+1}(n+1)-B_{k+1}(1))\quad\hbox{\cite{CFZ}}\,,
\label{ber1}
\end{align}
where Bernoulli numbers $B_n^{+}$ are determined by the recurrence formula
$$
\sum_{j=0}^k\binom{k+1}{j}B_j^{+}=k+1\quad (k\ge 0)
$$
or by the generating function
\begin{align}\label{defbernou}
\frac{t}{1-e^{-t}}=\sum_{n=0}^\infty B_n^{+}\frac{t^n}{n!}\,,
\end{align}
and  Bernoulli polynomials $B_n(x)$ are defined by the following generating function
$$
\frac{te^{xt}}{e^{t}-1}=\sum_{n=0}^\infty B_n(x)\frac{t^n}{n!}\,.
$$

\begin{Definition}\label{def1}
For $p\in \mathbb Z$ and $m, r, t \in \mathbb N$, define the quantities $S(p,m,t,r,0)$ and $S(p,m,1,r,1)$ as
\begin{align*}
&S(p,m,t,r,0):=\sum_{n=1}^\infty\frac{H_n^{(p)}}{n^{m}(n+r)^{t}}\,,\\
&S(p,m,t,r,1):=\sum_{n=1}^\infty\frac{(-1)^{n+1} H_n^{(p)}}{n^{m}(n+r)^{t}}\,.
\end{align*}
When $p \geq 0$, $H_n^{(-p)}$ is understood to be the sum $1^p+2^p+\cdots+n^p$.
\end{Definition}

\begin{Lem}[{\cite[Lemma 1.2]{Sofo}}]\label{lem1}
Let $s$ be a positive integer and $a>0$, then
\begin{align*}
\sum_{n=1}^\infty\frac{a H_n^{(s)}}{n(n+a)}
&=\zeta(s+1)+\sum_{j=1}^{a-1}\frac{(-1)^{s+1}H_j}{j^s}+\sum_{i=2}^s(-1)^{s-i}H_{a-1}^{(s-i+1)}\zeta(i)\,.
\end{align*}
\end{Lem}

\begin{Lem}\label{lem2}
Let $p, m, r \in \mathbb N$, then we have
\begin{align*}
S(p,m,1,r,0)
&=\sum_{i=2}^{m}\frac{(-1)^{m-i}}{r^{m-i+1}}S_{p,i}^{+,+}+\frac{(-1)^{m-1}}{r^{m}}\zeta(p+1)\\
&\quad +\frac{(-1)^{m-1}}{r^{m}}\left(\sum_{j=1}^{r-1}\frac{(-1)^{p+1}H_j}{j^p}
+\sum_{\ell=2}^{p}(-1)^{p-\ell}H_{r-1}^{(p-\ell+1)}\zeta(\ell)\right)\,.
\end{align*}
Let $m, r \in \mathbb N$, $p \in \mathbb N_{0}$ and $m \geq p+2$, then we have
\begin{align*}
&S(-p,m,1,r,0)\\
&=\frac{1}{p+1}\sum_{\ell=0}^{p} \binom{p+1}{\ell}B_{\ell}^{+}
\left(\sum_{i=2}^{m-p-1+\ell}\frac{(-1)^{m-p-1+\ell-i}}{r^{m-p+\ell-i}}\zeta(i)
+\frac{(-1)^{m-p-2+\ell}}{r^{m-p-1+\ell}} H_{r}\right)\,.
\end{align*}
\end{Lem}

\begin{proof}
When $p, m, r \in \mathbb N$, we can obtain that
\begin{align*}
\sum_{n=1}^\infty\frac{H_n^{(p)}}{n^{m}(n+r)}
&=\sum_{n=1}^\infty H_n^{(p)}\left(\sum_{i=2}^{m}\frac{(-1)^{m-i}}{r^{m-i+1}n^{i}}
+\frac{(-1)^{m-1}}{r^{m-1}n(n+r)}\right)\\
&=\sum_{i=2}^{m}\frac{(-1)^{m-i}}{r^{m-i+1}}\sum_{n=1}^\infty \frac{H_n^{(p)}}{n^{i}}
+\frac{(-1)^{m-1}}{r^{m-1}}\sum_{n=1}^\infty\frac{H_n^{(p)}}{n(n+r)}\,.
\end{align*}
With the help of Lemma {\ref{lem1}}, we get the desired result.

When $m, r \in \mathbb N$, $p \in \mathbb N_{0}$ and $m \geq p+2$, we have
\begin{align*}
\sum_{n=1}^\infty\frac{H_n^{(-p)}}{n^{m}(n+r)}
&=\sum_{n=1}^\infty\frac{\sum_{\ell=1}^{n}\ell^{p}}{n^{m}(n+r)}\\
&=\sum_{n=1}^\infty
\frac{\frac{1}{p+1}\sum_{\ell=0}^{p} \binom{p+1}{\ell}B_{\ell}^{+}n^{p+1-\ell}}{n^{m}(n+r)}\\
&=\frac{1}{p+1}\sum_{\ell=0}^{p} \binom{p+1}{\ell}B_{\ell}^{+} \sum_{n=1}^\infty \frac{1}{n^{m-p-1+\ell}(n+r)}\,.
\end{align*}
With the help of partial fraction expansion
\begin{align*}
\frac{1}{n^{t}(n+r)}
=\sum_{i=2}^{t}\frac{(-1)^{t-i}}{r^{t-i+1}}\cdot\frac{1}{n^{i}}
+\frac{(-1)^{t-1}}{r^{t-1}}\cdot\frac{1}{n(n+r)},
\end{align*}
we get the desired result.
\end{proof}

\begin{Lem}\label{lem3}
Let $p, r \in \mathbb N$, defining
\begin{align*}
S(p,r,1):=\sum_{n=1}^\infty \frac{(-1)^{n+1}H_n^{(p)}}{n+r}\,,
\end{align*}
then we have
\begin{align*}
&S(p,r,1)\\
&=(-1)^{r}S_{p,1}^{+,-}+(-1)^{r-1}\overline{\zeta}(p+1)
+\sum_{j=1}^{p}(-1)^{p-j+r}\overline{\zeta}(j)\overline{H}_{r-1}^{(p-j+1)}\\
&\quad +(-1)^{p+r-1}\overline{\zeta}(1)H_{r-1}^{(p)}
+(-1)^{p+r}\sum_{n=1}^{r-1}\frac{\overline{H}_{n}}{n^{p}}\,.
\end{align*}
\end{Lem}
\begin{proof}
By a change of counter, we have
\begin{align*}
&\quad S(p,r,1)\\
&=\sum_{n=1}^\infty \frac{(-1)^{n+1}H_n^{(p)}}{n+r}\\
&=\sum_{n=1}^\infty \frac{(-1)^{n}H_n^{(p)}}{n+r-1}-\sum_{n=1}^\infty \frac{(-1)^{n}}{n^{p}(n+r-1)}\\
&=-S(p,r-1,1)+\sum_{n=1}^\infty (-1)^{n+1}\left(\sum_{j=1}^{p}\frac{(-1)^{p-j}}{(r-1)^{p-j+1}}\cdot\frac{1}{n^{j}}
+\frac{(-1)^{p}}{(r-1)^{p}}\cdot\frac{1}{n+r-1}\right)\\
&=-S(p,r-1,1)+\sum_{j=1}^{p}\frac{(-1)^{p-j}}{(r-1)^{p-j+1}}\overline{\zeta}(j)
+\frac{(-1)^{p+r-1}}{(r-1)^{p}}(\overline{\zeta}(1)-\overline{H}_{r-1})\\
&=(-1)^{r-1}S(p,1,1)+\sum_{j=1}^{p}\overline{\zeta}(j)\sum_{n=1}^{r-1}\frac{(-1)^{p-j+r-1-n}}{n^{p-j+1}}\\
&\quad \quad +(-1)^{p+r-1}\overline{\zeta}(1)\sum_{n=1}^{r-1}\frac{1}{n^{p}}
+(-1)^{p+r}\sum_{n=1}^{r-1}\frac{\overline{H}_{n}}{n^{p}}\\
&=(-1)^{r-1}S(p,1,1)+\sum_{j=1}^{p}\overline{\zeta}(j)(-1)^{p-j+r}\overline{H}_{r-1}^{(p-j+1)}\\
&\quad \quad +(-1)^{p+r-1}\overline{\zeta}(1)H_{r-1}^{(p)}+(-1)^{p+r}\sum_{n=1}^{r-1}\frac{\overline{H}_{n}}{n^{p}}\,.
\end{align*}
Since
\begin{align*}
S(p,1,1)=\sum_{n=1}^\infty \frac{(-1)^{n+1}H_n^{(p)}}{n+1}=-S_{p,1}^{+,-}+\overline{\zeta}(p+1)\,,
\end{align*}
we get the desired result.
\end{proof}
Note that, $S(1,r,1)$ and $S(2,r,1)$ have already been obtained by Sofo (see \cite{Sofo1,Sofo2}).

\begin{Lem}\label{lem4}
Let $p, m, r \in \mathbb N$, then we have
\begin{align*}
&\quad S(p,m,1,r,1)\\
&=\sum_{i=1}^{m}\frac{(-1)^{m-i}}{r^{m-i+1}}S_{p,i}^{+,-}+\frac{(-1)^{m+r}}{r^{m}}S_{p,1}^{+,-}
+\frac{(-1)^{m+r-1}}{r^{m}}\overline{\zeta}(p+1)\\
&\quad +\frac{(-1)^{m}}{r^{m}}\left(\sum_{j=1}^{p}(-1)^{p-j+r}\overline{\zeta}(j)\overline{H}_{r-1}^{(p-j+1)}
+(-1)^{p+r-1}\overline{\zeta}(1)H_{r-1}^{(p)}\right)\\
&\quad +\frac{(-1)^{m+p+r}}{r^{m}}\sum_{n=1}^{r-1}\frac{\overline{H}_{n}}{n^{p}}\,.
\end{align*}
Let $m, r \in \mathbb N$, $p \in \mathbb N_{0}$ and $m \geq p+1$, then we have
{\scriptsize
\begin{align*}
&S(-p,m,1,r,1)\\
&=\frac{1}{p+1}\sum_{\ell=0}^{p} \binom{p+1}{\ell}B_{\ell}^{+}
\left(\sum_{i=1}^{m-p-1+\ell}\frac{(-1)^{m-p-1+\ell-i}}{r^{m-p+\ell-i}}\zeta(i)
+\frac{(-1)^{m-p-1+\ell+r}}{r^{m-p-1+\ell}}(\overline{\zeta}(1)-\overline{H}_{r})\right)\,.
\end{align*}
}
\end{Lem}
\begin{proof}
When $p, m, r \in \mathbb N$, we can obtain that
\begin{align*}
&\quad \sum_{n=1}^\infty\frac{(-1)^{n+1}H_n^{(p)}}{n^{m}(n+r)}\\
&=\sum_{n=1}^\infty (-1)^{n+1}H_n^{(p)}\left(\sum_{i=1}^{m}\frac{(-1)^{m-i}}{r^{m-i+1}n^{i}}
+\frac{(-1)^{m}}{r^{m}(n+r)}\right)\\
&=\sum_{i=1}^{m}\frac{(-1)^{m-i}}{r^{m-i+1}}\sum_{n=1}^\infty \frac{(-1)^{n+1}H_n^{(p)}}{n^{i}}
+\frac{(-1)^{m}}{r^{m}}\sum_{n=1}^\infty\frac{(-1)^{n+1}H_n^{(p)}}{n+r}\,.
\end{align*}
With the help of Lemma {\ref{lem3}}, we get the desired result.

When $m, r \in \mathbb N$, $p \geq 0$ and $m \geq p+1$, we have
\begin{align*}
\sum_{n=1}^\infty\frac{(-1)^{n+1}H_n^{(-p)}}{n^{m}(n+r)}
&=\sum_{n=1}^\infty\frac{(-1)^{n+1}\sum_{\ell=1}^{n}\ell^{p}}{n^{m}(n+r)}\\
&=\sum_{n=1}^\infty
\frac{(-1)^{n+1}\frac{1}{p+1}\sum_{\ell=0}^{p} \binom{p+1}{\ell}B_{\ell}^{+}n^{p+1-\ell}}{n^{m}(n+r)}\\
&=\frac{1}{p+1}\sum_{\ell=0}^{p} \binom{p+1}{\ell}B_{\ell}^{+} \sum_{n=1}^\infty \frac{(-1)^{n+1}}{n^{m-p-1+\ell}(n+r)}\,.
\end{align*}
With the help of partial fraction expansion
\begin{align*}
\frac{1}{n^{t}(n+r)}
=\sum_{i=1}^{t}\frac{(-1)^{t-i}}{r^{t-i+1}}\cdot\frac{1}{n^{i}}
+\frac{(-1)^{t}}{r^{t}}\cdot\frac{1}{n+r},
\end{align*}
we get the desired result.
\end{proof}

\begin{Lem}[\cite{Rusen}]\label{lem5}
For $r, n, p \in \mathbb N$, we have
\begin{align*}
H_n^{(p,r)}=\sum_{m=0}^{r-1} \sum_{j=0}^{r-1-m} a(r,m,j) n^{j} H_n^{(p-m)}\,.
\end{align*}
The coefficients $a(r,m,j)$ satisfy the following recurrence relations:
\begin{align*}
&a(r+1,r,0)=-\sum_{m=0}^{r-1} a(r,m,r-m-1)\frac{1}{r-m}\,,\\
&a(r+1,m,\ell)=\sum_{j=\ell-1}^{r-1-m} \frac{a(r,m,j)}{j+1} \binom{j+1}{j-\ell+1}B_{j-\ell+1}^{+}\\
&\qquad\qquad \qquad(0\leq m \leq r-1, 1\leq \ell \leq r-m)\,,\\
&a(r+1,m,0)=-\sum_{y=0}^{m} \sum_{j=max\{0, m-y-1\}}^{r-1-y}a(r,y,j)D(r,m,j,y)\quad (0\leq m \leq r-1)\,,
\end{align*}
where
$$
D(r,m,j,y)=\sum_{\ell=max\{0, m-y-1\}}^{j} \frac{1}{j+1} \binom{j+1}{j-\ell}B_{j-\ell}^{+}\binom{\ell+1}{m-y}(-1)^{1+\ell-m+y}\,.
$$
The initial value is given by $a(1,0,0)=1$.
\end{Lem}

Now we are able to prove our main theorems of this section.
\begin{theorem}\label{maintheorem1}
Let $s, p, m, k \in \mathbb N$ with $m \geq s$, then we have,
\begin{align*}
\sum_{n=1}^\infty\frac{H_n^{(p,s)}}{n^{m}\binom{n+k}{k}}
=\sum_{\ell_{1}=0}^{s-1} \sum_{\ell_{2}=0}^{s-1-\ell_{1}}a(s,\ell_{1},\ell_{2})
\sum_{r=1}^{k}(-1)^{r+1} r \binom{k}{r} S(p-\ell_{1},m-\ell_{2},1,r,0)\,,
\end{align*}
where $S(p-\ell_{1},m-\ell_{2},1,r,0)$ is given in Lemma \ref{lem2} and $a(s,\ell_{1},\ell_{2})$ is given in Lemma \ref{lem5}.
Therefore generalized hyperharmonic number sum
$$
\sum_{n=1}^\infty\frac{H_n^{(p,s)}}{n^{m}\binom{n+k}{k}}
$$
can be expressed in terms of classical Euler sums, zeta values and generalized harmonic numbers.
\end{theorem}
\begin{proof}
By using Lemma \ref{lem5}, we have
\begin{align*}
&\quad \sum_{n=1}^\infty\frac{H_n^{(p,s)}}{n^{m}\binom{n+k}{k}}\\
&=\sum_{\ell_{1}=0}^{s-1} \sum_{\ell_{2}=0}^{s-1-\ell_{1}}a(s,\ell_{1},\ell_{2})
\sum_{n=1}^\infty\frac{H_n^{(p-\ell_{1})}}{n^{m-\ell_{2}}\binom{n+k}{k}}\\
&=\sum_{\ell_{1}=0}^{s-1} \sum_{\ell_{2}=0}^{s-1-\ell_{1}}a(s,\ell_{1},\ell_{2})
\sum_{n=1}^\infty\frac{H_n^{(p-\ell_{1})}}{n^{m-\ell_{2}}}\sum_{r=1}^{k}(-1)^{r+1} r \binom{k}{r}\frac{1}{n+r}\\
&=\sum_{\ell_{1}=0}^{s-1} \sum_{\ell_{2}=0}^{s-1-\ell_{1}}a(s,\ell_{1},\ell_{2})\sum_{r=1}^{k}(-1)^{r+1} r \binom{k}{r}
\sum_{n=1}^\infty\frac{H_n^{(p-\ell_{1})}}{n^{m-\ell_{2}}(n+r)}\\
&=\sum_{\ell_{1}=0}^{s-1} \sum_{\ell_{2}=0}^{s-1-\ell_{1}}a(s,\ell_{1},\ell_{2})
\sum_{r=1}^{k}(-1)^{r+1} r \binom{k}{r} S(p-\ell_{1},m-\ell_{2},1,r,0)\,.
\end{align*}
\end{proof}

\begin{theorem}\label{maintheorem2}
Let $s, p, m, k \in \mathbb N$ with $m \geq s$, then we have,
\begin{align*}
&\quad \sum_{n=1}^\infty\frac{(-1)^{n+1}H_n^{(p,s)}}{n^{m}\binom{n+k}{k}}\\
&=\sum_{\ell_{1}=0}^{s-1} \sum_{\ell_{2}=0}^{s-1-\ell_{1}}a(s,\ell_{1},\ell_{2})
\sum_{r=1}^{k}(-1)^{r+1} r \binom{k}{r} S(p-\ell_{1},m-\ell_{2},1,r,1)\,,
\end{align*}
where $S(p-\ell_{1},m-\ell_{2},1,r,1)$ is given in Lemma \ref{lem4}.
Therefore generalized hyperharmonic number sum
$$
\sum_{n=1}^\infty\frac{(-1)^{n+1}H_n^{(p,s)}}{n^{m}\binom{n+k}{k}}
$$
can be expressed in terms of classical alternating Euler sums, zeta values and generalized (alternating) harmonic numbers.
\end{theorem}
\begin{proof}
By using Lemma \ref{lem5}, we have
\begin{align*}
&\quad \sum_{n=1}^\infty\frac{(-1)^{n+1}H_n^{(p,s)}}{n^{m}\binom{n+k}{k}}\\
&=\sum_{\ell_{1}=0}^{s-1} \sum_{\ell_{2}=0}^{s-1-\ell_{1}}a(s,\ell_{1},\ell_{2})
\sum_{n=1}^\infty\frac{(-1)^{n+1}H_n^{(p-\ell_{1})}}{n^{m-\ell_{2}}\binom{n+k}{k}}\\
&=\sum_{\ell_{1}=0}^{s-1} \sum_{\ell_{2}=0}^{s-1-\ell_{1}}a(s,\ell_{1},\ell_{2})
\sum_{n=1}^\infty\frac{(-1)^{n+1}H_n^{(p-\ell_{1})}}{n^{m-\ell_{2}}}\sum_{r=1}^{k}(-1)^{r+1} r \binom{k}{r}\frac{1}{n+r}\\
&=\sum_{\ell_{1}=0}^{s-1} \sum_{\ell_{2}=0}^{s-1-\ell_{1}}a(s,\ell_{1},\ell_{2})\sum_{r=1}^{k}(-1)^{r+1} r \binom{k}{r}
\sum_{n=1}^\infty\frac{(-1)^{n+1}H_n^{(p-\ell_{1})}}{n^{m-\ell_{2}}(n+r)}\\
&=\sum_{\ell_{1}=0}^{s-1} \sum_{\ell_{2}=0}^{s-1-\ell_{1}}a(s,\ell_{1},\ell_{2})
\sum_{r=1}^{k}(-1)^{r+1} r \binom{k}{r} S(p-\ell_{1},m-\ell_{2},1,r,1)\,.
\end{align*}
\end{proof}

\begin{Example}
Some illustrative examples are as following.

When $s=2, p=2, m=3, k=2$, we have
\begin{align*}
&\sum_{n=1}^\infty\frac{H_n^{(2,2)}}{n^{3}\binom{n+2}{2}}
=-\frac{9}{2}\zeta(5)+\frac{25}{4}\zeta(3)-\frac{17}{720}\pi^{4}-\frac{1}{4}\pi^{2}\,,\\
&\sum_{n=1}^\infty\frac{(-1)^{n+1}H_n^{(2,2)}}{n^{3}\binom{n+2}{2}}
=S_{2,3}^{+,-}-\frac{1}{2}S_{2,2}^{+,-}+\frac{3}{16}\zeta(3)-S_{1,3}^{+,-}
+\frac{3}{2}S_{1,2}^{+,-}-4S_{1,1}^{+,-}+\zeta(2)\,.
\end{align*}

When $s=2, p=1, m=3, k=2$, we have
\begin{align*}
\sum_{n=1}^\infty\frac{(-1)^{n+1}H_n^{(1,2)}}{n^{3}\binom{n+2}{2}}
&=S_{1,3}^{+,-}-\frac{1}{2}S_{1,2}^{+,-}-\frac{1}{2}\log{2}-\frac{7}{8}\zeta(2)+\frac{3}{2}\,,\\
&=-2Li_{4}(\frac{1}{2})+\frac{11}{4}\zeta(4)+\frac{1}{2}\zeta(2){\log{2}}^{2}-\frac{1}{12}(\log{2})^{4}\\
&\quad -\frac{7}{4}\zeta(3)\log{2}-\frac{5}{16}\zeta(3)-\frac{7}{8}\zeta(2)
-\frac{1}{2}\log{2}+\frac{3}{2}\,.
\end{align*}
In this expression we use the well-known polylogarithm function
$$
Li_{p}(x):=\sum_{n=1}^\infty \frac{x^n}{n^p}\quad (\lvert x \lvert \leq 1,\quad p \in \mathbb N)\,.
$$
\end{Example}

\section{Quadratic generalized hyperharmonic number sums}

In this section, we develop closed form representations for quadratic generalized hyperharmonic number sums with reciprocal binomial coefficients of types
$$
\sum_{n=1}^\infty\frac{H_n^{(p_{1},s_{1})}H_n^{(p_{2},s_{2})}}{n^{m}\binom{n+k}{k}}\,\quad\hbox{and}\quad
\sum_{n=1}^\infty\frac{(-1)^{n+1} H_n^{(p_{1},s_{1})}H_n^{(p_{2},s_{2})}}{n^{m}\binom{n+k}{k}}\,.
$$

Before going further, we introduce some notations and lemmata.

Following Flajolet-Salvy's paper \cite{Flajolet}, we write classical (alternating) quadratic Euler sums as
\begin{align*}
S_{p_{1},p_{2},q}^{+,+,+}:=\sum_{n=1}^\infty \frac{H_n^{(p_{1})}H_n^{(p_{2})}}{{n}^{q}}\quad\hbox{and}\quad
S_{p_{1},p_{2},q}^{+,+,-}:=\sum_{n=1}^\infty \frac{(-1)^{n-1}H_n^{(p_{1})}H_n^{(p_{2})}}{{n}^{q}}\,.
\end{align*}

\begin{Lem}[{{Abel's lemma on summation by parts} \cite{Abel,Chu}}]\label{lem6}
Let $\{f_k\}$ and $\{g_k\}$ be two sequences, and define the forward difference and backward difference, respectively, as
$$
\Delta\tau_k=\tau_{k+1}-\tau_k\quad\hbox{and}\quad \nabla\tau_k=\tau_k-\tau_{k-1}\,,
$$
then, there holds the relation:
\begin{align*}
\sum_{k=1}^\infty f_k\nabla g_k=\lim_{n\to\infty}f_n g_n-f_1 g_0-\sum_{k=1}^\infty g_k \Delta f_k \,.
\end{align*}
\end{Lem}

\begin{Lem}\label{lem7}
For $r, p_{1}, p_{2}\in \mathbb N$, we have
\begin{align*}
\sum_{n=1}^\infty \frac{r H_n^{(p_{1})} H_n^{(p_{2})}}{n(n+r)}
&=S_{p_{1},p_{2}+1}^{+,+}+S_{p_{2},p_{1}+1}^{+,+}-\zeta(p_{1}+p_{2}+1)
+\sum_{b=1}^{r-1}S(p_{1},p_{2},1,b,0)\\
&\quad +\sum_{b=1}^{r-1}S(p_{2},p_{1},1,b,0)
-\sum_{b=1}^{r-1}S(0,p_{1}+p_{2}+1,1,b,0)\,.
\end{align*}
\end{Lem}
\begin{proof}
Set
$$
f_{n}:=H_n^{(p_{1})} H_n^{(p_{2})}\quad\hbox{and}\quad
g_{n}:=\frac{1}{n+1}+\cdots+\frac{1}{n+r}\,,
$$
by using Lemma \ref{lem6}, we have
\begin{align*}
&\quad -\sum_{n=1}^\infty \frac{r H_n^{(p_{1})} H_n^{(p_{2})}}{n(n+r)}\\
&=\sum_{n=1}^\infty H_n^{(p_{1})} H_n^{(p_{2})}\left(\bigg(\frac{1}{n+1}+\cdots+\frac{1}{n+r}\bigg)-\bigg(\frac{1}{n}+\cdots+\frac{1}{n+r-1}\bigg)\right)\\
&=-\sum_{n=0}^\infty \bigg(\frac{1}{n+1}+\cdots+\frac{1}{n+r}\bigg)
\bigg(\frac{H_n^{(p_{1})}}{(n+1)^{p_{2}}}+\frac{H_n^{(p_{2})}}{(n+1)^{p_{1}}}+\frac{1}{(n+1)^{p_{1}+p_{2}}}\bigg)\\
&=-\sum_{n=0}^\infty \sum_{b=0}^{r-1}\frac{1}{n+1+b}
\bigg(\frac{H_{n+1}^{(p_{1})}}{(n+1)^{p_{2}}}+\frac{H_{n+1}^{(p_{2})}}{(n+1)^{p_{1}}}-\frac{1}{(n+1)^{p_{1}+p_{2}}}\bigg)\\
&=-\sum_{b=0}^{r-1}\sum_{n=1}^\infty \frac{1}{n+b}
\bigg(\frac{H_{n}^{(p_{1})}}{n^{p_{2}}}+\frac{H_{n}^{(p_{2})}}{n^{p_{1}}}-\frac{1}{n^{p_{1}+p_{2}}}\bigg)\\
&=-\sum_{n=1}^\infty
\bigg(\frac{H_{n}^{(p_{1})}}{n^{p_{2}+1}}+\frac{H_{n}^{(p_{2})}}{n^{p_{1}+1}}-\frac{1}{n^{p_{1}+p_{2}+1}}\bigg)
-\sum_{b=1}^{r-1}\sum_{n=1}^\infty \frac{H_{n}^{(p_{1})}}{n^{p_{2}}(n+b)}\\
&\quad -\sum_{b=1}^{r-1}\sum_{n=1}^\infty \frac{H_{n}^{(p_{2})}}{n^{p_{1}}(n+b)}+\sum_{b=1}^{r-1}\sum_{n=1}^\infty \frac{1}{n^{p_{1}+p_{2}}(n+b)}\\
&=-S_{p_{1},p_{2}+1}^{+,+}-S_{p_{2},p_{1}+1}^{+,+}+\zeta(p_{1}+p_{2}+1)
-\sum_{b=1}^{r-1}S(p_{1},p_{2},1,b,0)\\
&\quad -\sum_{b=1}^{r-1}S(p_{2},p_{1},1,b,0)
+\sum_{b=1}^{r-1}S(0,p_{1}+p_{2}+1,1,b,0)\,.
\end{align*}
\end{proof}

\begin{Definition}\label{def2}
For $p_{1}, p_{2}\in \mathbb Z$ and $m, r, t \in \mathbb N$, define the quantities $T(p_{1},p_{2},m,t,r,0)$ and $T(p_{1},p_{2},m,t,r,1)$ as
\begin{align*}
&T(p_{1},p_{2},m,t,r,0):=\sum_{n=1}^\infty\frac{H_n^{(p_{1})}H_n^{(p_{2})}}{n^{m}(n+r)^{t}}\,,\\
&T(p_{1},p_{2},m,t,r,1):=\sum_{n=1}^\infty\frac{(-1)^{n+1} H_n^{(p_{1})}H_n^{(p_{2})}}{n^{m}(n+r)^{t}}\,.
\end{align*}
When $p \geq 0$, $H_n^{(-p)}$ is understood to be the sum $1^p+2^p+\cdots+n^p$.
\end{Definition}

\begin{Lem}\label{lem8}
Let $p_{1}, p_{2}, m, r \in \mathbb N$, then we have
\begin{align*}
&\quad T(p_{1},p_{2},m,1,r,0)\\
&=\sum_{i=2}^{m}\frac{(-1)^{m-i}}{r^{m-i+1}}S_{p_{1},p_{2},i}^{+,+,+}
+\frac{(-1)^{m-1}}{r^{m}}\left(S_{p_{1},p_{2}+1}^{+,+}+S_{p_{2},p_{1}+1}^{+,+}\right)\\
&\quad -\frac{(-1)^{m-1}}{r^{m}}\zeta(p_{1}+p_{2}+1) +\frac{(-1)^{m-1}}{r^{m}}\sum_{b=1}^{r-1}S(p_{1},p_{2},1,b,0)\\
&\quad +\frac{(-1)^{m-1}}{r^{m}}\left(\sum_{b=1}^{r-1}S(p_{2},p_{1},1,b,0)
-\sum_{b=1}^{r-1}S(0,p_{1}+p_{2}+1,1,b,0)\right)\,.
\end{align*}
Let $p_{1}, m, r \in \mathbb N$, $p_{2} \in \mathbb N_{0}$ and $m \geq p_{2}+2$, then we have
\begin{align*}
&T(p_{1},-p_{2},m,1,r,0)\\
&=\frac{1}{p_{2}+1}\sum_{\ell=0}^{p_{2}} \binom{p_{2}+1}{\ell}B_{\ell}^{+}S(p_{1},m-p_{2}-1+\ell,1,r,0)\,.
\end{align*}
Let $m, r \in \mathbb N$, $p_{1}, p_{2} \in \mathbb N_{0}$ and $m \geq p_{1}+p_{2}+3$, then we have
\begin{align*}
T(-p_{1},-p_{2},m,1,r,0)&=\frac{1}{(p_{1}+1)(p_{2}+1)}\sum_{\ell_{1}=0}^{p_{1}} \sum_{\ell_{2}=0}^{p_{2}} \binom{p_{1}+1}{\ell_{1}}\binom{p_{2}+1}{\ell_{2}}\\
&\quad \times B_{\ell_{1}}^{+}B_{\ell_{2}}^{+}S(0,m-p_{1}-p_{2}-1+\ell_{1}+\ell_{2},1,r,0)\,.
\end{align*}
\end{Lem}

\begin{proof}
When $p_{1}, p_{2}, m, r \in \mathbb N$, we can obtain that
\begin{align*}
\sum_{n=1}^\infty\frac{H_n^{(p_{1})}H_n^{(p_{2})}}{n^{m}(n+r)}
&=\sum_{n=1}^\infty H_n^{(p_{1})}H_n^{(p_{2})}\left(\sum_{i=2}^{m}\frac{(-1)^{m-i}}{r^{m-i+1}n^{i}}
+\frac{(-1)^{m-1}}{r^{m-1}n(n+r)}\right)\\
&=\sum_{i=2}^{m}\frac{(-1)^{m-i}}{r^{m-i+1}}\sum_{n=1}^\infty \frac{H_n^{(p_{1})}H_n^{(p_{2})}}{n^{i}}
+\frac{(-1)^{m-1}}{r^{m-1}}\sum_{n=1}^\infty\frac{H_n^{(p_{1})}H_n^{(p_{2})}}{n(n+r)}\,.
\end{align*}
With the help of Lemma {\ref{lem7}}, we get the desired result.

When $p_{1}, m, r \in \mathbb N$, $p_{2} \in \mathbb N_{0}$ and $m \geq p_{2}+2$, we have
\begin{align*}
\sum_{n=1}^\infty\frac{H_n^{(p_{1})}H_n^{(-p_{2})}}{n^{m}(n+r)}
&=\sum_{n=1}^\infty\frac{H_n^{(p_{1})}\sum_{\ell=1}^{n}\ell^{p_{2}}}{n^{m}(n+r)}\\
&=\sum_{n=1}^\infty
\frac{H_n^{(p_{1})}\frac{1}{p_{2}+1}\sum_{\ell=0}^{p_{2}} \binom{p_{2}+1}{\ell}B_{\ell}^{+}n^{p_{2}+1-\ell}}{n^{m}(n+r)}\\
&=\frac{1}{p_{2}+1}\sum_{\ell=0}^{p_{2}} \binom{p_{2}+1}{\ell}B_{\ell}^{+} \sum_{n=1}^\infty \frac{H_n^{(p_{1})}}{n^{m-p_{2}-1+\ell}(n+r)}\,.
\end{align*}
With the help of Lemma {\ref{lem2}}, we get the desired result.

When $m, r \in \mathbb N$, $p_{1}, p_{2} \in \mathbb N_{0}$ and $m \geq p_{1}+p_{2}+3$, we have
\begin{align*}
&\quad \sum_{n=1}^\infty\frac{H_n^{(-p_{1})}H_n^{(-p_{2})}}{n^{m}(n+r)}\\
&=\sum_{n=1}^\infty\frac{\sum_{\ell_{1}=1}^{n}\ell^{p_{1}}\sum_{\ell_{2}=1}^{n}\ell^{p_{2}}}{n^{m}(n+r)}\\
&=\sum_{n=1}^\infty
\frac{\frac{1}{p_{1}+1}\sum_{\ell=0}^{p_{1}} \binom{p_{1}+1}{\ell}B_{\ell}^{+}n^{p_{1}+1-\ell}\frac{1}{p_{2}+1}\sum_{\ell=0}^{p_{2}} \binom{p_{2}+1}{\ell}B_{\ell}^{+}n^{p_{2}+1-\ell}}{n^{m}(n+r)}\\
&=\frac{1}{(p_{1}+1)(p_{2}+1)}\sum_{\ell_{1}=0}^{p_{1}} \sum_{\ell_{2}=0}^{p_{2}} \binom{p_{1}+1}{\ell_{1}}\binom{p_{2}+1}{\ell_{2}}\\
&\quad \times B_{\ell_{1}}^{+}B_{\ell_{2}}^{+} \sum_{n=1}^\infty \frac{1}{n^{m-p_{1}-p_{2}-2+\ell_{1}+\ell_{2}}(n+r)}\,.
\end{align*}
With the help of Lemma {\ref{lem2}}, we get the desired result.
\end{proof}

\begin{Lem}\label{lem9}
Let $p_{1}, p_{2}, r \in \mathbb N$, defining
\begin{align*}
T(p_{1},p_{2},r):=\sum_{n=1}^\infty \frac{(-1)^{n+1}H_n^{(p_{1})}H_n^{(p_{2})}}{n+r}\,,
\end{align*}
then we have
\begin{align*}
&T(p_{1},p_{2},r)\\
&=(-1)^{r}\left(S_{p_{1},p_{2},1}^{+,+,-}-S_{p_{1},p_{2}+1}^{+,-}-S_{p_{2},p_{1}+1}^{+,-}+\overline{\zeta}(p_{1}+p_{2}+1)\right)\\
&\quad +\sum_{j=1}^{r-1}(-1)^{r-1-j}\big(S(p_{1},p_{2},1,j,1)+S(p_{2},p_{1},1,j,1)\big)\\
&\quad +\sum_{j=1}^{r-1}(-1)^{r-j}S(0,p_{1}+p_{2}+1,1,j,1)\,.
\end{align*}
\end{Lem}
\begin{proof}
By a change of counter, we have
\begin{align*}
&\quad T(p_{1},p_{2},r)\\
&=\sum_{n=1}^\infty \frac{(-1)^{n+1}H_n^{(p_{1})}H_n^{(p_{2})}}{n+r}\\
&=\sum_{n=1}^\infty \frac{(-1)^{n}H_n^{(p_{1})}H_n^{(p_{2})}}{n+r-1}+\sum_{n=1}^\infty \frac{(-1)^{n+1}H_n^{(p_{1})}}{n^{p_{2}}(n+r-1)}+\sum_{n=1}^\infty \frac{(-1)^{n+1}H_n^{(p_{2})}}{n^{p_{1}}(n+r-1)}\\
&\quad +\sum_{n=1}^\infty \frac{(-1)^{n}}{n^{p_{1}+p_{2}}(n+r-1)}\\
&=-T(p_{1},p_{2},r-1)+S(p_{1},p_{2},1,r-1,1)+S(p_{2},p_{1},1,r-1,1)\\
&\quad -S(0,p_{1}+p_{2}+1,1,r-1,1)\\
&=(-1)^{r-1}T(p_{1},p_{2},1)+\sum_{j=1}^{r-1}(-1)^{r-1-j}S(p_{1},p_{2},1,j,1)\\
&\quad +\sum_{j=1}^{r-1}(-1)^{r-1-j}S(p_{2},p_{1},1,j,1)
+\sum_{j=1}^{r-1}(-1)^{r-j}S(0,p_{1}+p_{2}+1,1,j,1)\,.
\end{align*}
Since
\begin{align*}
T(p_{1},p_{2},1)
&=\sum_{n=1}^\infty \frac{(-1)^{n+1}H_n^{(p_{1})}H_n^{(p_{2})}}{n+1}\\
&=\sum_{n=1}^\infty \frac{(-1)^{n}H_n^{(p_{1})}H_n^{(p_{2})}}{n}+\sum_{n=1}^\infty \frac{(-1)^{n+1}H_n^{(p_{1})}}{n^{p_{2}+1}}\\
&\quad +\sum_{n=1}^\infty \frac{(-1)^{n+1}H_n^{(p_{2})}}{n^{p_{1}+1}}+\sum_{n=1}^\infty \frac{(-1)^{n}}{n^{p_{1}+p_{2}+1}}\\
&=-S_{p_{1},p_{2},1}^{+,+,-}+S_{p_{1},p_{2}+1}^{+,-}+S_{p_{2},p_{1}+1}^{+,-}-\overline{\zeta}(p_{1}+p_{2}+1)\,,
\end{align*}
we get the desired result.
\end{proof}
Note that, $T(1,1,1)$ has already been obtained by Sofo \cite{Sofo1}.

\begin{Lem}\label{lem10}
Let $p_{1}, p_{2}, r \in \mathbb N$ and $m \in \mathbb N_{0}$, then we have
\begin{align*}
&\quad T(p_{1},p_{2},m,1,r,1)\\
&=\sum_{i=1}^{m}\frac{(-1)^{m-i}}{r^{m-i+1}}S_{p_{1},p_{2},i}^{+,+,-}
+\frac{(-1)^{m+r}}{r^{m}}\left(S_{p_{1},p_{2},1}^{+,+,-}-S_{p_{1},p_{2}+1}^{+,-}-S_{p_{2},p_{1}+1}^{+,-}\right)\\
&\quad +\frac{(-1)^{m+r}}{r^{m}}\overline{\zeta}(p_{1}+p_{2}+1)
+\sum_{j=1}^{r-1}\frac{(-1)^{r+m-j}}{r^{m}}S(0,p_{1}+p_{2}+1,1,j,1)\\
&\quad +\sum_{j=1}^{r-1}\frac{(-1)^{r+m-1-j}}{r^{m}}\big(S(p_{1},p_{2},1,j,1)+S(p_{2},p_{1},1,j,1)\big)\,.
\end{align*}
Let $p_{1}, m, r \in \mathbb N$, $p_{2} \in \mathbb N_{0}$ and $m \geq p_{2}+1$, then we have
\begin{align*}
&T(p_{1},-p_{2},m,1,r,1)\\
&=\frac{1}{p_{2}+1}\sum_{\ell=0}^{p_{2}} \binom{p_{2}+1}{\ell}B_{\ell}^{+}S(p_{1},m-p_{2}-1+\ell,1,r,1)\,.
\end{align*}
Let $m, r \in \mathbb N$, $p_{1}, p_{2} \in \mathbb N_{0}$ and $m \geq p_{1}+p_{2}+2$, then we have
\begin{align*}
T(-p_{1},-p_{2},m,1,r,1)&=\frac{1}{(p_{1}+1)(p_{2}+1)}\sum_{\ell_{1}=0}^{p_{1}} \sum_{\ell_{2}=0}^{p_{2}} \binom{p_{1}+1}{\ell_{1}}\binom{p_{2}+1}{\ell_{2}}\\
&\quad \times B_{\ell_{1}}^{+}B_{\ell_{2}}^{+}S(0,m-p_{1}-p_{2}-1+\ell_{1}+\ell_{2},1,r,1)\,.
\end{align*}
\end{Lem}
\begin{proof}
When $p_{1}, p_{2}, r \in \mathbb N$ and $m \in \mathbb N_{0}$, we can obtain that
\begin{align*}
&\quad \sum_{n=1}^\infty\frac{(-1)^{n+1}H_n^{(p_{1})}H_n^{(p_{2})}}{n^{m}(n+r)}\\
&=\sum_{n=1}^\infty (-1)^{n+1}H_n^{(p_{1})}H_n^{(p_{2})}\left(\sum_{i=1}^{m}\frac{(-1)^{m-i}}{r^{m-i+1}n^{i}}
+\frac{(-1)^{m}}{r^{m}(n+r)}\right)\\
&=\sum_{i=1}^{m}\frac{(-1)^{m-i}}{r^{m-i+1}}\sum_{n=1}^\infty \frac{(-1)^{n+1}H_n^{(p_{1})}H_n^{(p_{2})}}{n^{i}}
+\frac{(-1)^{m}}{r^{m}}\sum_{n=1}^\infty\frac{(-1)^{n+1}H_n^{(p_{1})}H_n^{(p_{2})}}{n+r}\,.
\end{align*}
With the help of Lemma {\ref{lem9}}, we get the desired result.

When $p_{1}, m, r \in \mathbb N$, $p_{2} \in \mathbb N_{0}$ and $m \geq p_{2}+1$, we have
\begin{align*}
\sum_{n=1}^\infty\frac{(-1)^{n+1}H_n^{(p_{1})}H_n^{(-p_{2})}}{n^{m}(n+r)}
&=\sum_{n=1}^\infty\frac{(-1)^{n+1}H_n^{(p_{1})}\sum_{\ell=1}^{n}\ell^{p_{2}}}{n^{m}(n+r)}\\
&=\sum_{n=1}^\infty
\frac{(-1)^{n+1}H_n^{(p_{1})}\frac{1}{p_{2}+1}\sum_{\ell=0}^{p_{2}} \binom{p_{2}+1}{\ell}B_{\ell}^{+}n^{p_{2}+1-\ell}}{n^{m}(n+r)}\\
&=\frac{1}{p_{2}+1}\sum_{\ell=0}^{p_{2}} \binom{p_{2}+1}{\ell}B_{\ell}^{+} \sum_{n=1}^\infty \frac{(-1)^{n+1}H_n^{(p_{1})}}{n^{m-p_{2}-1+\ell}(n+r)}\,.
\end{align*}
With the help of Lemma {\ref{lem4}}, we get the desired result.

When $m, r \in \mathbb N$, $p_{1}, p_{2} \in \mathbb N_{0}$ and $m \geq p_{1}+p_{2}+2$, we have
\begin{align*}
&\quad \sum_{n=1}^\infty\frac{(-1)^{n+1}H_n^{(-p_{1})}H_n^{(-p_{2})}}{n^{m}(n+r)}\\
&=\sum_{n=1}^\infty\frac{(-1)^{n+1}\sum_{\ell_{1}=1}^{n}\ell^{p_{1}}\sum_{\ell_{2}=1}^{n}\ell^{p_{2}}}{n^{m}(n+r)}\\
&=\sum_{n=1}^\infty
\frac{(-1)^{n+1}\frac{1}{p_{1}+1}\sum_{\ell=0}^{p_{1}} \binom{p_{1}+1}{\ell}B_{\ell}^{+}n^{p_{1}+1-\ell}\frac{1}{p_{2}+1}\sum_{\ell=0}^{p_{2}} \binom{p_{2}+1}{\ell}B_{\ell}^{+}n^{p_{2}+1-\ell}}{n^{m}(n+r)}\\
&=\frac{1}{(p_{1}+1)(p_{2}+1)}\sum_{\ell_{1}=0}^{p_{1}} \sum_{\ell_{2}=0}^{p_{2}} \binom{p_{1}+1}{\ell_{1}}\binom{p_{2}+1}{\ell_{2}}\\
&\quad \times B_{\ell_{1}}^{+}B_{\ell_{2}}^{+} \sum_{n=1}^\infty \frac{(-1)^{n+1}}{n^{m-p_{1}-p_{2}-2+\ell_{1}+\ell_{2}}(n+r)}\,.
\end{align*}
With the help of Lemma {\ref{lem4}}, we get the desired result.
\end{proof}

Now we are able to prove our main theorems of this section.
\begin{theorem}\label{maintheorem3}
Let $s_{1}, s_{2}, p_{1}, p_{2}, m, k \in \mathbb N$ with $m \geq s_{1}+s_{2}-1$, then we have
\begin{align*}
&\quad \sum_{n=1}^\infty\frac{H_n^{(p_{1},s_{1})}H_n^{(p_{2},s_{2})}}{n^{m}\binom{n+k}{k}}\\
&=\sum_{\ell_{1}=0}^{s_{1}-1}\sum_{t_{1}=0}^{s_{1}-1-\ell_{1}}
\sum_{\ell_{2}=0}^{s_{2}-1}\sum_{t_{2}=0}^{s_{2}-1-\ell_{2}}a(s_{1},\ell_{1},t_{1})a(s_{2},\ell_{2},t_{2})\\
&\quad \times \sum_{r=1}^{k}(-1)^{r+1} r \binom{k}{r} T(p_{1}-\ell_{1},p_{2}-\ell_{2},m-t_{1}-t_{2},1,r,0)\,,
\end{align*}
where $a(s,\ell_{x},t_{x}), x=1,2$ are given in Lemma \ref{lem5} and $T(p_{1}-\ell_{1},p_{2}-\ell_{2},m-t_{1}-t_{2},1,r,0)$ is given in Lemma \ref{lem8}.
Therefore generalized hyperharmonic number sum
$$
 \sum_{n=1}^\infty\frac{H_n^{(p_{1},s_{1})}H_n^{(p_{2},s_{2})}}{n^{m}\binom{n+k}{k}}
$$
can be expressed in terms of classical Euler sums, zeta values and generalized harmonic numbers.
\end{theorem}
\begin{proof}
By using Lemma \ref{lem5}, we have
\begin{align*}
&\quad \sum_{n=1}^\infty\frac{H_n^{(p_{1},s_{1})}H_n^{(p_{2},s_{2})}}{n^{m}\binom{n+k}{k}}\\
&=\sum_{\ell_{1}=0}^{s_{1}-1}\sum_{t_{1}=0}^{s_{1}-1-\ell_{1}}
\sum_{\ell_{2}=0}^{s_{2}-1}\sum_{t_{2}=0}^{s_{2}-1-\ell_{2}}a(s_{1},\ell_{1},t_{1})a(s_{2},\ell_{2},t_{2})
\sum_{n=1}^\infty\frac{H_n^{(p_{1}-\ell_{1})}H_n^{(p_{2}-\ell_{2})}}{n^{m-t_{1}-t_{2}}\binom{n+k}{k}}\\
&=\sum_{\ell_{1}=0}^{s_{1}-1}\sum_{t_{1}=0}^{s_{1}-1-\ell_{1}}
\sum_{\ell_{2}=0}^{s_{2}-1}\sum_{t_{2}=0}^{s_{2}-1-\ell_{2}}a(s_{1},\ell_{1},t_{1})a(s_{2},\ell_{2},t_{2})\\
&\quad \times \sum_{n=1}^\infty\frac{H_n^{(p_{1}-\ell_{1})}H_n^{(p_{2}-\ell_{2})}}{n^{m-t_{1}-t_{2}}}\sum_{r=1}^{k}(-1)^{r+1} r \binom{k}{r}\frac{1}{n+r}\\
&=\sum_{\ell_{1}=0}^{s_{1}-1}\sum_{t_{1}=0}^{s_{1}-1-\ell_{1}}
\sum_{\ell_{2}=0}^{s_{2}-1}\sum_{t_{2}=0}^{s_{2}-1-\ell_{2}}a(s_{1},\ell_{1},t_{1})a(s_{2},\ell_{2},t_{2})\\
&\quad \times \sum_{r=1}^{k}(-1)^{r+1} r \binom{k}{r} \sum_{n=1}^\infty\frac{H_n^{(p_{1}-\ell_{1})}H_n^{(p_{2}-\ell_{2})}}{n^{m-t_{1}-t_{2}}(n+r)}\\
&=\sum_{\ell_{1}=0}^{s_{1}-1}\sum_{t_{1}=0}^{s_{1}-1-\ell_{1}}
\sum_{\ell_{2}=0}^{s_{2}-1}\sum_{t_{2}=0}^{s_{2}-1-\ell_{2}}a(s_{1},\ell_{1},t_{1})a(s_{2},\ell_{2},t_{2})\\
&\quad \times \sum_{r=1}^{k}(-1)^{r+1} r \binom{k}{r}T(p_{1}-\ell_{1},p_{2}-\ell_{2},m-t_{1}-t_{2},1,r,0)\,.
\end{align*}
\end{proof}

\begin{theorem}\label{maintheorem4}
Let $s_{1}, s_{2}, p_{1}, p_{2}, m, k \in \mathbb N$ with $m \geq s_{1}+s_{2}-1$, then we have
\begin{align*}
&\quad \sum_{n=1}^\infty\frac{(-1)^{n+1}H_n^{(p_{1},s_{1})}H_n^{(p_{2},s_{2})}}{n^{m}\binom{n+k}{k}}\\
&=\sum_{\ell_{1}=0}^{s_{1}-1}\sum_{t_{1}=0}^{s_{1}-1-\ell_{1}}
\sum_{\ell_{2}=0}^{s_{2}-1}\sum_{t_{2}=0}^{s_{2}-1-\ell_{2}}a(s_{1},\ell_{1},t_{1})a(s_{2},\ell_{2},t_{2})\\
&\quad \times \sum_{r=1}^{k}(-1)^{r+1} r \binom{k}{r} T(p_{1}-\ell_{1},p_{2}-\ell_{2},m-t_{1}-t_{2},1,r,1)\,,
\end{align*}
where $a(s,\ell_{x},t_{x}), x=1,2$ are given in Lemma \ref{lem5} and $T(p_{1}-\ell_{1},p_{2}-\ell_{2},m-t_{1}-t_{2},1,r,1)$ is given in Lemma \ref{lem10}.
Therefore generalized hyperharmonic number sum
$$
 \sum_{n=1}^\infty\frac{(-1)^{n+1}H_n^{(p_{1},s_{1})}H_n^{(p_{2},s_{2})}}{n^{m}\binom{n+k}{k}}
$$
can be expressed in terms of classical (alternating) Euler sums, zeta values and generalized (alternating) harmonic numbers.
\end{theorem}
\begin{proof}
By using Lemma \ref{lem5}, we have
\begin{align*}
&\quad \sum_{n=1}^\infty\frac{(-1)^{n+1}H_n^{(p_{1},s_{1})}H_n^{(p_{2},s_{2})}}{n^{m}\binom{n+k}{k}}\\
&=\sum_{\ell_{1}=0}^{s_{1}-1}\sum_{t_{1}=0}^{s_{1}-1-\ell_{1}}
\sum_{\ell_{2}=0}^{s_{2}-1}\sum_{t_{2}=0}^{s_{2}-1-\ell_{2}}a(s_{1},\ell_{1},t_{1})a(s_{2},\ell_{2},t_{2})
\sum_{n=1}^\infty\frac{(-1)^{n+1}H_n^{(p_{1}-\ell_{1})}H_n^{(p_{2}-\ell_{2})}}{n^{m-t_{1}-t_{2}}\binom{n+k}{k}}\\
&=\sum_{\ell_{1}=0}^{s_{1}-1}\sum_{t_{1}=0}^{s_{1}-1-\ell_{1}}
\sum_{\ell_{2}=0}^{s_{2}-1}\sum_{t_{2}=0}^{s_{2}-1-\ell_{2}}a(s_{1},\ell_{1},t_{1})a(s_{2},\ell_{2},t_{2})\\
&\quad \times \sum_{n=1}^\infty\frac{(-1)^{n+1}H_n^{(p_{1}-\ell_{1})}H_n^{(p_{2}-\ell_{2})}}{n^{m-t_{1}-t_{2}}}\sum_{r=1}^{k}(-1)^{r+1} r \binom{k}{r}\frac{1}{n+r}\\
&=\sum_{\ell_{1}=0}^{s_{1}-1}\sum_{t_{1}=0}^{s_{1}-1-\ell_{1}}
\sum_{\ell_{2}=0}^{s_{2}-1}\sum_{t_{2}=0}^{s_{2}-1-\ell_{2}}a(s_{1},\ell_{1},t_{1})a(s_{2},\ell_{2},t_{2})\\
&\quad \times \sum_{r=1}^{k}(-1)^{r+1} r \binom{k}{r} \sum_{n=1}^\infty\frac{(-1)^{n+1}H_n^{(p_{1}-\ell_{1})}H_n^{(p_{2}-\ell_{2})}}{n^{m-t_{1}-t_{2}}(n+r)}\\
&=\sum_{\ell_{1}=0}^{s_{1}-1}\sum_{t_{1}=0}^{s_{1}-1-\ell_{1}}
\sum_{\ell_{2}=0}^{s_{2}-1}\sum_{t_{2}=0}^{s_{2}-1-\ell_{2}}a(s_{1},\ell_{1},t_{1})a(s_{2},\ell_{2},t_{2})\\
&\quad \times \sum_{r=1}^{k}(-1)^{r+1} r \binom{k}{r}T(p_{1}-\ell_{1},p_{2}-\ell_{2},m-t_{1}-t_{2},1,r,1)\,.
\end{align*}
\end{proof}

\section{Some interesting integrals}
De Doelder \cite{Doelder} gave the following integral:
\begin{align*}
\int_{0}^{\frac{\pi}{2}}\frac{\phi^{2}}{\sin \phi}\mathrm{d}\phi=-\frac{7}{2}\zeta(3)+2\pi G\,,
\end{align*}
where $G$ is the famous Catalan's constant defined as
$$
G:=\sum_{n=1}^\infty \frac{(-1)^{n-1}}{(2n-1)^{2}}\,.
$$

Consider the complex function $f(z)=\log^{2} z/(z^{2}-1)$ and to integrate $f$ in positive sense along the contour given by $0 < \delta \leq x \leq 1$; $z=e^{i\delta}, 0 \leq \phi \leq \frac{\pi}{2}$; $1 \geq y \geq \delta >0$ and $z=\delta e^{i\phi}, \frac{\pi}{2} \geq \phi \geq 0$.

Within this contour there are no singularities of $f$ and by the Cauchy residue theorem we have
\begin{align*}
&\lim_{\delta \to 0}\bigg(-\int_{\delta}^{1}\frac{\log^{2} x}{1-x^2}\mathrm{d}x
-\int_{0}^{\frac{\pi}{2}}\frac{\phi^{2}}{2\sin \phi}\mathrm{d}\phi
+i\int_{\delta}^{1}\frac{(\log y+\frac{1}{2}\pi i)^{2}}{1+y^2}\mathrm{d}y\\
&\qquad \qquad\qquad\qquad\qquad\qquad\qquad\qquad+i\int_{\frac{\pi}{2}}^{0}\frac{(\log \delta+i\phi)^{2}}{(\delta e^{i\phi})^{2}-1}\delta e^{i\phi}\mathrm{d}\phi\bigg)=0\,,
\end{align*}
Comparing the real part and the imaginary part on both sides, we have
\begin{align*}
&\int_{0}^{\frac{\pi}{2}}\frac{\phi^{2}}{\sin \phi}\mathrm{d}\phi
=2\bigg(\int_{0}^{1}\frac{\log^{2} x}{x^2-1}\mathrm{d}x-\pi \int_{0}^{1}\frac{\log y}{1+y^2}\mathrm{d}y\bigg)\,,\\
&\int_{0}^{1}\frac{\log^{2} y}{1+y^2}\mathrm{d}y=\frac{\pi^2}{4}\int_{0}^{1}\frac{1}{1+y^2}\mathrm{d}y=\frac{\pi^3}{16}\,.
\end{align*}
It is known (see \cite{Doelder}) that $\int_{0}^{1}\frac{\log^{2} x}{x^2-1}\mathrm{d}x=-\frac{7}{4}\zeta(3)$ and $\int_{0}^{1}\frac{\log y}{1+y^2}=-G$, so we get the valuation of the integral $\int_{0}^{\frac{\pi}{2}}\frac{\phi^{2}}{\sin \phi}\mathrm{d}\phi$.

De Doelder \cite{Doelder} also considered the function $g(z)=\log z/(z^{2}-1)$ along the same contour. Then the following results could be established:
\begin{align*}
&\int_{0}^{\frac{\pi}{2}}\frac{\phi}{\sin \phi}\mathrm{d}\phi=-2\int_{0}^{1}\frac{\log y}{1+y^2}\mathrm{d}y
=2G\,,\\
&\int_{0}^{1}\frac{\log x}{1+x^2}\mathrm{d}x=\frac{\pi}{2}\int_{0}^{1}\frac{1}{1+y^2}\mathrm{d}y=\frac{\pi^2}{8}\,.
\end{align*}

We now consider the function $f(z)=-\log z/(2-z)$ along the same contour, since within this contour there are no singularities, by using the Cauchy residue theorem we can obtain that
\begin{align*}
&\lim_{\delta \to 0}\bigg(-\int_{\delta}^{1}\frac{\log x}{2-x}\mathrm{d}x
+\int_{0}^{\frac{\pi}{2}}\frac{\phi e^{i\phi}}{2-e^{i\phi}}\mathrm{d}\phi
+\int_{\delta}^{1}\frac{i(2\log y-\frac{\pi}{2}y)-(\pi+y\log y)}{4+y^2}\\
&\qquad\qquad\qquad\qquad\qquad\qquad\qquad\qquad\qquad+\int_{0}^{\frac{\pi}{2}}\frac{i\delta e^{i\phi}(\log \delta+i\phi)}{2-\delta e^{i\phi}}\mathrm{d}\phi\bigg)=0\,.
\end{align*}
It follows that
\begin{align*}
&\int_{0}^{\frac{\pi}{2}}\frac{\phi(2\cos \phi-1)}{5-4\sin \phi}\mathrm{d}\phi=\int_{0}^{1}\frac{\log x}{2-x}\mathrm{d}x+\int_{0}^{1}\frac{\pi +y\log y}{4+y^2}\mathrm{d}y\,,\\
&\int_{0}^{\frac{\pi}{2}}\frac{2\phi\sin \phi}{5-4\sin \phi}\mathrm{d}\phi=\int_{0}^{1}\frac{\frac{\pi}{2}y-2\log y}{4+y^2}\mathrm{d}y\,.
\end{align*}
By a change of variable, we have
\begin{align*}
&\int_{0}^{1}\frac{\log x}{2-x}\mathrm{d}x=\int_{0}^{1}\frac{\log (1-x)}{1+x}\mathrm{d}x=\frac{1}{2}\log^{2}2-\frac{1}{2}\zeta(2)\quad (\cite[p.153]{Sofo1})\,,\\
&\int_{0}^{1}\frac{\pi}{4+y^2}\mathrm{d}y=\frac{\pi}{2}\int_{0}^{\frac{1}{2}}\frac{1}{1+x^2}\mathrm{d}x
=\frac{\pi}{2}\arctan{\frac{1}{2}}\,,\\
&\int_{0}^{1}\frac{y\log y}{4+y^2}\mathrm{d}y=\int_{0}^{\frac{1}{2}}\frac{y(\log y+\log 2)}{1+y^2}\mathrm{d}y\\
&\qquad \qquad \qquad=\frac{1}{2}\log2\log\frac{5}{4}+\log \frac{1}{2}\sum_{n=1}^\infty\frac{(-1)^{n-1}(\frac{1}{2})^{2n}}{2n}
-\sum_{n=1}^\infty\frac{(-1)^{n-1}(\frac{1}{2})^{2n}}{4n^2}\\
&\qquad \qquad \qquad=\frac{1}{4}Li_{2}(-\frac{1}{4})\,,\\
&\int_{0}^{1}\frac{\frac{\pi}{2}y}{4+y^2}\mathrm{d}y
=\frac{\pi}{4}\int_{0}^{1}\frac{1}{4+y}\mathrm{d}y=\frac{\pi}{4}\log\frac{5}{4}\,,\\
&\int_{0}^{1}\frac{2\log y}{4+y^2}\mathrm{d}y=\int_{0}^{\frac{1}{2}}\frac{\log y+\log 2}{1+y^2}\mathrm{d}y\\
&\qquad \qquad \qquad =\log\frac{1}{2}\arctan\frac{1}{2}-\int_{0}^{\frac{1}{2}}\frac{\arctan y}{y}\mathrm{d}y+\log 2\arctan\frac{1}{2}\\
&\qquad \qquad \qquad =-Ti_{2}(\frac{1}{2})\,,
\end{align*}
where we have used the inverse tangent integral $Ti_{2}(x):=\int_{0}^{x}\frac{\arctan y}{y}\mathrm{d}y$.

Combining the above results, we have the following proposition:
\begin{Prop}
\begin{align*}
&\int_{0}^{\frac{\pi}{2}}\frac{\phi(2\cos \phi-1)}{5-4\sin \phi}\mathrm{d}\phi=-\frac{1}{12}\pi^{2}+\frac{1}{2}\log^{2}2+\frac{\pi}{2}\arctan{\frac{1}{2}}+\frac{1}{4}Li_{2}(-\frac{1}{4})\,,\\
&\int_{0}^{\frac{\pi}{2}}\frac{2\phi\sin \phi}{5-4\sin \phi}\mathrm{d}\phi=\frac{\pi}{4}\log\frac{5}{4}+Ti_{2}(\frac{1}{2})\,.
\end{align*}
\end{Prop}

\section{Some formulas for harmonic numbers}

In this section, we develop some formulas for harmonic numbers in terms of binomial coefficients. We begin by recalling a known result for harmonic numbers \cite{Sofo1}.

For $n \in \mathbb N_{0}$, the following result holds:
\begin{align*}
-\frac{H_{n+1}}{n+1}=\int_{0}^{1}y^{n}\log y\mathrm{d}y\,.
\end{align*}
We are now going to prove our main result of this section.

\begin{Lem}\label{lem11}
Let $n, m \in \mathbb N_{0}$, defining
\begin{align*}
L(n, m, x):=\int_{0}^{x}y^{n}\log^{m} y \mathrm{d}y\,,
\end{align*}
then we have
\begin{align*}
L(n, m, x)=\frac{x^{n+1}}{n+1}\sum_{j=0}^{m}\frac{(m+1-j)_{j}}{(n+1)^j}(-1)^{j}\log^{m-j} x\,,
\end{align*}
where $(t)_{n}=t(t+1)\cdots(t+n-1)$ is the Pochhammer symbol. In particular, we have $L(n, m, 1)=\frac{m!(-1)^m}{(n+1)^{m+1}}$.
\end{Lem}
\begin{proof}
\begin{align*}
L(n, m, x)&=\frac{x^{n+1}}{n+1}\log^{m} x-\frac{m}{n+1}\int_{0}^{x}y^{n}\log^{m-1} y \mathrm{d}y\\
&=\frac{x^{n+1}}{n+1}\log^{m} x-\frac{m}{n+1}L(n, m-1, x)\\
&=\frac{x^{n+1}}{n+1}\log^{m} x-\frac{m x^{n+1}}{(n+1)^2}\log^{m-1} x+\frac{m(m-1)}{(n+1)^2}L(n, m-2, x)\\
&=\frac{x^{n+1}}{n+1}\sum_{j=0}^{m}\frac{(m+1-j)_{j}}{(n+1)^j}(-1)^{j}\log^{m-j} x\,.
\end{align*}
\end{proof}

\begin{Lem}\label{lem12}
Let $n, m \in \mathbb N_{0}$, defining
\begin{align*}
M(n, m, x):=\int_{x}^{1}y^{n}\log^{m} (1-y) \mathrm{d}y\,,
\end{align*}
then we have
\begin{align*}
M(n, m, x)=\sum_{j=0}^{n}\binom{n}{j}(-1)^{j}\frac{(1-x)^{j+1}}{j+1}
\sum_{i=0}^{m}\frac{(m+1-i)_{i}}{(j+1)^i}(-1)^{i}\log^{m-i} (1-x)\,.
\end{align*}
In particular, we have $M(n, m, 0)=(-1)^{m}m!\sum_{j=0}^{n}\binom{n}{j}\frac{(-1)^{j}}{(j+1)^{m+1}}$.
\end{Lem}
\begin{proof}
By a change of variable, we have
\begin{align*}
M(n, m, x)&=\int_{0}^{1-x}(1-t)^{n}\log^{m} t \mathrm{d}t\\
&=\sum_{j=0}^{n}\binom{n}{j}(-1)^{j}\int_{0}^{1-x}t^{j}\log^{m} t \mathrm{d}t\,.
\end{align*}
With the help of Lemma\ref{lem11}, we get the desired result.
\end{proof}
Note that, $-\frac{H_{n+1}}{n+1}=M(n, 1, 0)$, then we have the following proposition:
\begin{Prop}\label{prop2}
For $n\in \mathbb N_{0}$ and $r \in \mathbb N$, we have
\begin{align*}
H_{n+1}&=(n+1)\sum_{j=0}^{n}\binom{n}{j}\frac{(-1)^{j}}{(j+1)^{2}}\\
&=\sum_{j=0}^{n}\binom{n+1}{j+1}\frac{(-1)^{j}}{j+1}\,,\\
h_n^{(r)}&=\binom{n+r-1}{r-1}\bigg(\sum_{j_{1}=0}^{n+r-2}\binom{n+r-1}{j_{1}+1}\frac{(-1)^{j_{1}}}{j_{1}+1}
-\sum_{j_{2}=0}^{r-2}\binom{r-1}{j_{2}+1}\frac{(-1)^{j_{2}}}{j_{2}+1}\bigg)\,.
\end{align*}
\end{Prop}
The following formulas are known \cite{Devoto}:
\begin{align*}
&M(n, 2, 0)=\frac{2}{n+1}\bigg(H_{n+1}^{(2)}+\sum_{k=1}^{n}\frac{H_{k}}{k+1}\bigg)\,,\\
&M(n, 3, 0)=-\frac{6}{n+1}\bigg(H_{n+1}^{(3)}+\sum_{j=1}^{n}\frac{H_{j}}{(j+1)^2}
+\sum_{j=1}^{n}\frac{H_{j}^{(2)}}{j+1}+\sum_{k=1}^{n}\frac{1}{k+1}\sum_{j=1}^{k-1}\frac{H_{j}}{j+1}\bigg)\,.
\end{align*}
With the help of Proposition \ref{prop2}, we have the following proposition:
\begin{Prop}
For $n\in \mathbb N_{0}$, we have
\begin{align*}
&H_{n+1}^{(2)}=\sum_{j=0}^{n}\binom{n+1}{j+1}\frac{(-1)^{j}}{(j+1)^2}
-\sum_{k=0}^{n-1}\frac{1}{k+2}\sum_{j=0}^{k}\binom{k+1}{j+1}\frac{(-1)^{j}}{j+1}\,,\\
&H_{n+1}^{(3)}=\sum_{j=0}^{n}\binom{n+1}{j+1}\frac{(-1)^{j}}{(j+1)^3}
-\sum_{j=0}^{n-1}\frac{1}{(j+1)^2}\sum_{\ell=0}^{j}\binom{j+1}{\ell+1}\frac{(-1)^{\ell}}{\ell+1}\\
&\qquad \qquad -\sum_{j=0}^{n-1}\frac{1}{j+2}\sum_{\ell=0}^{j}\binom{j+1}{\ell+1}\frac{(-1)^{\ell}}{(\ell+1)^2}\,.
\end{align*}
\end{Prop}

\end{document}